\documentclass[11pt,twoside]{amsart}
\usepackage{tabularx}
\usepackage{geometry}
\usepackage{amssymb,amsmath,amsthm,latexsym}
\usepackage{mathrsfs}
\usepackage{url}
\usepackage{amsfonts}
\usepackage{graphicx}
\usepackage{hyperref}

\usepackage{graphicx}
\usepackage{amssymb,amsthm,amsmath}
\usepackage{mathtools}

\newtheorem{theorem}{Theorem}[section]

\usepackage{enumerate}

\usepackage{hyperref}
\usepackage{xcolor}

\hypersetup{colorlinks=true,linkcolor=blue,urlcolor=red}
\vspace{5cm}
  
\begin{document}

\title{Divisibility and Sequence Properties of  $\sigma^+$ and $\varphi^+$} %%%%%%%%%%%%
\author[Sagar Mandal]{Sagar Mandal}
\address{Department of Mathematics\\Indian Institute of Technology Ropar, Punjab, India.}
\email{sagarmandal31415@gmail.com}

\maketitle
\let\thefootnote\relax
\footnotetext{\it Keywords and phrases: Euler's totient function, Unitary totient function, Schemmel’s totient function, Jordan totient function, Sum of positive divisors function.}
\let\thefootnote\relax
\footnotetext{\it MSC2020:  11A25} %%%%%%%%%%
\maketitle

\begin{abstract}
Inspired by Lehmer’s and Deaconescu’s conjectures, as well as various analogue problems concerning Euler’s totient function $\varphi(n)$, Schemmel’s totient function $S_{2}(n)$, Jordan totient function $J_k$, and the unitary totient function $\varphi^{*}(n)$, we investigate analogous divisibility problems involving the functions $\sigma(n)$, $\sigma^{+}(n)$, and $\varphi^{+}(n)$.
 Further, we establish some interesting properties of the sequences $\left\{\sigma^+(n)\right\}_{n=1}^\infty$ and $\left\{\varphi^+(n)\right\}_{n=1}^\infty$, in particular, we prove that each of these sequences contains infinitely many arithmetic progressions of length $3$.\\
{\bf Keywords:}  Euler's totient function, Unitary totient function, Schemmel’s totient function, Jordan totient function, Sum of positive divisors function.

\end{abstract}

%\tableofcontents
\section{Introduction}
Lehmer \cite{10} conjectured that if $\varphi(n)\mid (n-1)$, where $\varphi(n)$ is the Euler’s totient function, then $n$ must be a prime number. A positive composite integer $n$ that satisfies the condition $\varphi(n)\mid (n-1)$ is called a Lehmer number. Let $S_2(n)$ denote Schemmel’s totient function, a multiplicative function defined by
$$
S_2(p^k) :=
\begin{cases}
0, & \text{if } p = 2 \\
p^{k-1} (p - 2), & \text{if } p > 2
\end{cases}
$$
where $p$ is a prime number and $k\in \mathbb{Z}^{+}$. Deaconescu \cite{5} considered Schemmel’s totient function and conjectured that for $n\geq 2$
$$S_2(n)\mid \varphi(n)-1$$
if and only if $n$ is prime. A positive composite integer $n$ is said to be a Deaconescu number if $S_2(n)\mid \varphi(n)-1$. Both of these conjectures remain unproven, interesting properties of Lehmer numbers can be found in~\cite{1,2,3,7,9,11,12,15}, and interesting properties of Deaconescu numbers can be found in~\cite{8,13}. Besides these problems, many analogues of Lehmer’s problem have been studied, for example, the analogue of Lehmer’s problem for the unitary totient function $\varphi^{*}$, namely the problem $\varphi^{*}(n)\mid (n-1)$, and for the Jordan totient function $J_k$, namely the problem $J_k(n)\mid (n^k-1)$, have been investigated in \cite{18}. The divisibility condition $(\varphi(n)+1)\mid n$ was studied in \cite{4}, and the problem $\varphi(n)^2 \mid (n^2 - 1)$ was examined and completely resolved in \cite{11}. For further analogues of Lehmer’s problem, interested readers may consult \cite{6,16}.
In a recent paper, the author~\cite{14} revisits the functions 
$$\varphi^{+}(n)  \ \  =\ \  \prod_{p\mid n}(\varphi(p^{v_{p}(n)})+1),\quad \varphi^+(1)=1 $$
and
$$\sigma^{+}(n)  \ \  =\ \  \prod_{p\mid n}(\sigma(p^{v_{p}(n)})+1),\quad \sigma^+(1)=1$$
where $v_{p}(n)$ denotes the highest power of a prime number $p$ dividing $n$, which were originally introduced in~\cite{17}. In this note, he investigates the behavior of 
$$
\sum_{\substack{n \leq x \\ \omega(n)=2}} \left(\varphi^+ - \varphi\right)\left(n\right)
\quad \text{and} \quad
\sum_{\substack{n \leq x \\ \omega(n)=2}} \left(\sigma^+- \sigma\right)\left(n\right),
$$
for all real $x \ge 6$. Inspired by these conjectures and problems, we investigate whether a similar statement holds for $n+1\mid \sigma^+(n)$ and $\varphi^+(n)\mid (n-1)$. Later, we explore and establish several interesting properties of the sequences $\left\{\sigma^+(n)\right\}_{n=1}^\infty$ and $\left\{\varphi^+(n)\right\}_{n=1}^\infty$.
\section{Main Results}
Before we begin proving results related to $\sigma^+(n)$, let us first explore some basic divisibility properties of the classical sum-of-divisors function $\sigma(n)$. Note that if $n=p$, where $p$ is a prime, then $n+1\mid \sigma(n)$ as $\sigma(n)=p+1$. We now attempt to determine whether there exists any positive composite integer $n$ such that 
$n+1\mid \sigma(n),$ which is equivalent to saying that there exists some positive integer $k$ such that $\sigma(n)=k(n+1)$. It is easy to note that
if $n=pq$ for some prime numbers $p,q$ then $n+1\nmid \sigma(n)$ as if  $n+1\mid \sigma(n)$ then $\sigma(n)=k(n+1)$ for some positive integer $k$. Then $k\geq2$ as for $k=1$ clearly $n+1\nmid \sigma(n)$. Note that
$\sigma(pq)=k(pq+1)$ implies,
$\sigma(p)\sigma(q)=k(pq+1)$, since $\sigma(p)=p+1,\sigma(q)=q+1$ we have
$pq+p+q+1=kpq+k$
that is
$(pq+1)(k-1)=p+q.$
Since 
$$\frac{p+q}{pq+1}<\frac{p+q}{pq}=\frac{1}{p}+\frac{1}{q}<1$$
we have
$$k=1+\frac{p+q}{pq+1}<1+1=2$$
a contradiction. Now let us consider the case $n=p^2q$, where $p,q$ are prime numbers. Also consider $k=2$. Since $\sigma(p^2)=1+p+p^2$ we have
$$\sigma(p^2q)=(1+p+p^2)(q+1)=2(p^2q+1)$$
which is equivalent to
\begin{align}\label{1}
    q=1+\frac{2p}{p^2-p-1}.
\end{align}

Now choose $p=2$ to get $q=5$ and thus we have a desired $n=2^2\cdot 5=20$ so that $n+1\mid\sigma(n)$. 
\begin{theorem}\label{thm2}
If $\sigma(n)=2(n+1)$ holds for $n=p^2q$ for some primes $p,q$ then $p=2,q=5$ is the only solution.
\end{theorem}
\begin{proof}
   From (\ref{1}) we get
  $$q=1+\frac{2p}{p^2-p-1}=\frac{p^2+p-1}{p^2-p-1},$$
  since $q\in \mathbb{N}$ we have
  $$p^2+p-1\equiv 0 \pmod{p^2-p-1}$$
  that is
  $$2p\equiv 0 \pmod{p^2-p-1}.$$
  It follows that
  $$p^2-p-1\leq 2p$$
  in other words
  $$p^2-3p-1\leq 0.$$
Consider $f(x)=x^2-3x-1$, then the zeros of $f(x)$ are $\frac{3\pm\sqrt{13}}{2}$. Thus, $f(x)\leq 0$ for $x\in [\frac{3-\sqrt{13}}{2},\frac{3+\sqrt{13}}{2}]$, it follows that $p^2-3p-1\leq 0$ only for $p=2,3$. Now note that for $p=3$ we have from (\ref{1}) that
$$1+\frac{6}{5}=q\not\in\mathbb{N}.$$
Hence $\sigma(n)=2(n+1)$ holds for $n=p^2q$ when $p=2,q=5$. This completes the proof.
\end{proof}
Since, we get a solution for the divisibility problem $n+1\mid\sigma(n)$, this naturally leads us to the following open question:\\
\textbf{Open question 1}: Are there infinitely many positive composite integers $n$ such that
$n+1\mid \sigma(n)?$\\\\
Let us now consider the corresponding problem for $\sigma^+(n)$. Note that for any prime $p$ we have $\sigma^+(p)=p+2$, it follows that $p+1 \nmid \sigma^+(p)$, makes the problem more interesting. Let us try to deduce a result similar to the case when $n=pq$ and $n=p^2q$, where $p,q$ are prime numbers.
\begin{theorem}
 If $n=pq$ where $p,q$ are prime numbers, then $n+1\nmid \sigma^+(n)$.   
\end{theorem}
\begin{proof}
Let $n=pq$ for some prime numbers $p,q$. If possible, suppose that $\sigma^+(n)=k(n+1)$ for some positive integer $k$. Then $k\geq2$ as for $k=1$ clearly $n+1\nmid \sigma(n)$. Note that
$$\sigma^+(pq)=k(pq+1)$$
which implies
$$(p+2)(q+2)=k(pq+1)$$
from which we can write
\begin{align}\label{2}
  (pq+1)(k-1)=2(p+q)+3.  
\end{align}

Since 
$$\frac{2(p+q)+3}{pq+1}<\frac{2(p+q)+3}{pq}=\frac{2}{p}+\frac{2}{q}+\frac{3}{pq}\leq 1+\frac{2}{3}+\frac{1}{2}<3$$
we have $k<4$. Now for $k=2$, from (\ref{2}) we get
\begin{align}\label{3}
  pq+1=2p+2q+3  
\end{align}

which implies that $2\mid pq$, without loss of generality, set $p=2$, then from (\ref{3}) we get $6=0$ a contradiction. Now for $k=3$, we get from (\ref{2}) that
  $2\mid 2p+2q+3$, implies $2\mid 3$, again a contradiction. Therefore, for any positive integer $k$, $\sigma^+(n)=k(n+1)$ does not hold. This completes the proof.
\end{proof}

\begin{theorem}
For any $n=p^2q$ where $p,q$ are prime numbers, $\sigma^+(n)=2(n+1)$ does not hold.
\end{theorem}
\begin{proof}
Note that
$$\sigma^+(n)=\sigma^+(p^2)\sigma^+(q)=(2+p+p^2)(2+q)=2(p^2q+1)$$
implies that
\begin{align}\label{4}
q=\frac{2(p^2+p+1)}{p^2-p-2}=2+\frac{4p+6}{p^2-p-2}.  
\end{align}
Also observe that $q-2\in \mathbb{N}$ as if $q-2=0$, i.e., $q=2$ then from (\ref{4}) we get $p=-\frac{3}{2}$. It follows that $p^2-p-2\mid 4p+6$, therefore
  $$p^2-p-2\leq 4p+6$$
  in other words
  $$p^2-5p-8\leq 0.$$
Consider $f(x)=x^2-5x-8$, then the zeros of $f(x)$ are $\frac{5\pm \sqrt{57}}{2}$. Thus, $f(x)\leq 0$ for all $x\in [\frac{5-\sqrt{57}}{2},\frac{5+ \sqrt{57}}{2}]$, it follows that $f(p)=p^2-5p-8\leq 0$ only for $p=2,3,5$. Now note that $p$ cannot be $2$ from (\ref{4}). Also, for $p=3,5$ we get from (\ref{4}) that $q=13/2,31/9\not\in \mathbb{N}$ respectively. Thus, for any prime numbers $p,q$, $2(p^2q+1)=\sigma^+(p^2q)$ does not hold. This completes the proof.
\end{proof}
\begin{theorem}
If $n=p_1^{r_1}\cdots p_{r}^{r_r}$, where $p_j$ are prime numbers, and $r_j\geq 1$ then 
$$\frac{1}{2}I(n)<\frac{\sigma^+(n)}{n+1}<2^rI(n),$$
where $I(n)$ is the abundancy index of $n$.
\end{theorem}
\begin{proof}
Let $n=p_1^{r_1}\cdots p_{r}^{r_r}$, where $p_j$ are prime numbers, and $r_j\geq 1$ then
$$\frac{1}{2}I(n)=\frac{\sigma(n)}{2n}<\frac{\sigma(n)}{n+1}<\frac{\sigma^+(n)}{n+1}.$$
Also, since
$$\sigma^+(n)=\prod_{j=1}^{r}\bigg(\sigma(p_j^{r_j})+1\bigg)<\prod_{j=1}^{r}\bigg(\sigma(p_j^{r_j})+\sigma(p_j^{r_j})\bigg)=2^r\prod_{j=1}^{r}\sigma(p_j^{r_j})=2^r\sigma(n)$$
we have
$$\frac{\sigma^+(n)}{n+1}<\frac{\sigma^+(n)}{n}<\frac{2^r\sigma(n)}{n}=2^rI(n).$$
This completes the proof.
\end{proof}
We have checked all positive composite integers $n$ from $1$ to $10^5$ but none of them satisfied the condition that $\sigma^+(n)$ is divisible by $n+1$. Therefore, we are offering some open questions,\\\\
\textbf{Open question 2}: Is there any positive composite integer $n$ such that $n+1\mid \sigma^+(n)$? \\\\
\textbf{Open question 3}: Are there infinitely many positive composite integers $n$ such that $n+1\mid \sigma^+(n)$? \\\\
We now establish some interesting properties of the sequence $\left\{\sigma^+(n)\right\}_{n=1}^\infty$.

\begin{theorem}
 In the sequence $\left\{\sigma^+(n)\right\}_{n=1}^\infty$, there exist infinitely many pairs $(n, m) \in \mathbb{N}^2$ with $n \ne m$ such that $\sigma^+(n) = \sigma^+(m)$.
\end{theorem}
\begin{proof}
Let $n=2^k\cdot 5$ and $m=2^{k-1}\cdot 9$, where $k\geq 2$, then
$\sigma^+(n)=\sigma^+(2^k)\cdot\sigma^+(5)=2^{k+1}\cdot 7=2^{k}\cdot 14=\sigma^+(2^{k-1})\cdot\sigma^+(9)=\sigma^+(2^{k-1}\cdot9)=\sigma^+(m)$. Since the number of such values of $k$ is infinite, so are the corresponding pairs $(n, m)$ satisfying $\sigma^+(n) = \sigma^+(m)$.

\end{proof}
\begin{theorem}
The sequence $\left\{\sigma^+(n)\right\}_{n=1}^\infty$ contains infinitely many arithmetic progressions of length $3$.
   
\end{theorem}
\begin{proof}
  Let $a=2^k$, $b=2^{k+2}$ and $c=2^{k+1}\cdot 5$, where $k\geq 1$. To prove the theorem, it suffices to show that $2\sigma^+(b) = \sigma^+(a) + \sigma^+(c)$. The infinitude of arithmetic progressions of length $3$ then follows from the infinitude of such values of $k$. We can see that
  $$2\sigma^+(b)=2^{k+4}=2^{k+1}+2^{k+1}\cdot7=\sigma^+(2^k)+\sigma^+(2^k\cdot 5)=\sigma^+(a)+\sigma^+(c).$$
  This completes the proof.
\end{proof}
Inspired by the above theorem, we propose the following conjecture:\\
\textbf{Conjecture 1:}\label{c1} For each prime $q \geq 3$, there exist infinitely many pairs $(p, n)$ such that $\sigma^+(n) = 2p - q$, where $p$ is a prime number.\\
We now proceed to derive analogous results for the function $\varphi^+(n)$. Note that if $n=p$, where $p$ is a prime, then $\varphi^+(n)\nmid (n-1)$ as $\varphi^+(n)=\varphi(n)+1=(p-1)+1=p$. We now attempt to determine whether there exists any positive composite integer $n$ such that 
$\varphi^+(n)\mid (n-1),$
which is equivalent to saying that there exists some positive integer $k$ such that $(n-1)=k\varphi^+(n)$.

\begin{theorem}\label{thm1}
 If $n=p_1p_2\cdots p_r$, where $p_i$ are prime numbers, then $\varphi^+(n)\nmid (n-1)$.   
\end{theorem}
\begin{proof}
Let $n=p_1p_2\cdots p_r$, where $p_i$ are prime numbers. If possible, suppose that $(n-1)=k\varphi^+(n)$ for some positive integer $k$. Then
$$k\varphi^+(n)=kp_1p_2\cdots p_r=p_1p_2\cdots p_r-1,$$
that is
$$k=\frac{p_1p_2\cdots p_r-1}{p_1p_2\cdots p_r}<1$$
a contradiction. Therefore $\varphi^+(n)\nmid (n-1)$.
\end{proof}
Let us consider the case $n=p^2$, where $p$ is a prime number. Also consider $k=1$. Since $\varphi^+(p^2)=p^2-p+1$, we have
$$p^2-1=p^2-p+1$$
which gives $p=2$ and thus we get a positive composite integer $n=4$ such that $\varphi^+(n)\mid (n-1)$. Now we prove that if $n=p^r$ for which $\varphi^+(n)\mid (n-1)$ then $r=2$ and $p=2$.
\begin{theorem}
If $\varphi^+(n)\mid (n-1)$, where $n=p^r$, $p$ is a prime number, then $r=2$ and $p=2$.    
\end{theorem}
\begin{proof}
Let $p^r-1=k\varphi^+(p^r)$ holds for some positive integer $k$. If possible, suppose that $k\geq2$. Then 
$$p^r-1=k\varphi^+(p^r
)=k(p^r-p^{r-1}+1)$$
which is equivalent to
$$k=\frac{p^r-1}{p^r-p^{r-1}+1}<\frac{p^r}{p^{r-1}(p-1)}=\frac{p}{p-1}\leq2$$
that is $k<2$, which is a contradiction. Therefore, we must have $k=1$, which implies that $p^r-1=\varphi^+(p^r)=p^r-p^{r-1}+1$, which implies that $p^{r-1}=2$, which is true when $r=2$ and $p=2$. This completes the proof.
\end{proof}

\begin{theorem}
If $n=p_1^{r_1}\cdots p_{r}^{r_r}$, where $p_j$ are prime numbers, and $r_j\geq 2$ then 
$$\prod_{j=1}^r\biggl(1+\frac{1}{p_j^2}\biggl)\leq \frac{n-1}{\varphi^+(n)}<2^r.$$
\end{theorem}
\begin{proof}
We first prove that, for each $p_j$ the following holds
$$\biggl(1+\frac{1}{p_j^2}\biggl)(p_j^{r_j}-p_j^{r_j-1}+1)<p_j^{r_j},$$
which is equivalent to prove that
$$p_j^{r_j+2}-(1+p_j^2)(p_j^{r_j}-p_j^{r_j-1}+1)>0.$$
Note that for $p_j\geq2$ and $r_j\geq 2$, we have
$$p_j^{r_j+2}-(1+p_j^2)(p_j^{r_j}-p_j^{r_j-1}+1)=p_j^{r_j}(p_j-1)+p_j^{r_j-1}-p_j^2-1>0.$$
Therefore
$$\prod_{j=1}^r\biggl(1+\frac{1}{p_j^2}\biggl)(p_j^{r_j}-p_j^{r_j-1}+1)<\prod_{j=1}^{r}p_j^{r_j}$$
that is
$$\prod_{j=1}^r\biggl(1+\frac{1}{p_j^2}\biggl)(p_j^{r_j}-p_j^{r_j-1}+1)\leq\prod_{j=1}^{r}p_j^{r_j}-1=n-1,$$
since $\prod_{j=1}^r(p_j^{r_j}-p_j^{r_j-1}+1)=\varphi^+(n)$, we have
$$\prod_{j=1}^r\biggl(1+\frac{1}{p_j^2}\biggl)\leq\frac{n-1}{\varphi^+(n)}.$$
Also, note that, for each $p_j$, we have
$$\frac{p_j^{r_j}}{p_j^{r_j}-p^{r_j-1}_{j}}=\frac{p_j}{p_j-1}\leq2$$
then
$$\frac{n-1}{\varphi^+(n)}<\frac{n}{\varphi^+(n)}<\prod_{j=1}^r\frac{p_j^{r_j}}{p_j^{r_j}-p^{r_j-1}_{j}}\leq 2^{r}.$$
This completes the proof.
\end{proof}
We have verified, through exhaustive computation , that there is only one positive integer $n\leq 10^5$ such that
  $$\varphi^+(n)\mid (n-1).$$
Based on this empirical observation, we propose the following two open questions:\\
\textbf{Open question 4:} Is there any positive composite integer $n\neq 4$ such that $\varphi^+(n)\mid (n-1)$?\\
\textbf{Open question 5:} Are there infinitely many positive composite integers $n$ such that
$\varphi^+(n)\mid (n-1)?$\\\\
We now establish similar interesting properties for the sequence $\left\{\varphi^+(n)\right\}_{n=1}^\infty$.

\begin{theorem}
   In the sequence $\left\{\varphi^+(n)\right\}_{n=1}^\infty$, there exist infinitely many pairs $(n, m) \in \mathbb{N}^2$ with $n \ne m$ such that $\varphi^+(n) = \varphi^+(m)$.
\end{theorem}
\begin{proof}
Let $n=7p$ and $m=9p$, where $p$ is a prime other than $3,7$, then
$\varphi^+(n)=7p=(6+1)p=(\varphi(3^2)+1)p=\varphi^+(9)p=\varphi^+(9p)=\varphi^+(m)$. Since the number of prime $p$ is infinite, so are the corresponding pairs $(n, m)$ satisfying $\varphi^+(n) = \varphi^+(m)$.
\end{proof}
\begin{theorem}
    The sequence $\left\{\varphi^+(n)\right\}_{n=1}^\infty$ contains infinitely many arithmetic progressions of length $3$.
\end{theorem}
\begin{proof}
Let $a=3p$, $b=7p$ and $c=11p$, where $p$ is a prime number greater than or equal to $13$. To prove the theorem, it suffices to show that $2\varphi^+(b) = \varphi^+(a) + \varphi^+(c)$. The infinitude of arithmetic progressions of length $3$ then follows from the infinitude of prime $p$. We can observe that
$$2\varphi^+(b)=14p=3p+11p=\varphi^+(a)+\varphi^+(c),$$
this completes the proof.
\end{proof}
We have a conjecture that, for infinitely many primes $p$, $\varphi(p+1) > \varphi(p)$, this conjecture is still open. We give a similar result for $\varphi^+(n)$.
\begin{theorem}
For all primes $p\in$\textup{\href{https://oeis.org/A005382}{OEIS A005382} }and $p>2$, we have $\varphi^+(q)<\varphi^+(q+1)$, where $q=2p-1$ is a prime.
\end{theorem}
\begin{proof}
Since $q=2p-1$ is a prime, we have $\varphi^+(q)=2p-1$ and $\varphi^+(q+1)=\varphi^+(2p)=2p$, thus $\varphi^+(q)<\varphi^+(q+1)$.
\end{proof}

\textbf{Open question 6:} Are there infinitely many prime $p$ such that $\varphi^+(p)<\varphi^+(p+1)$?\\
We end this paper with an analogous conjecture to Conjecture 1:\\
\textbf{Conjecture 2:} For each prime $q\geq 3$, there are infinitely many pairs $(p,n)$, such that $\varphi^+(n)=2p-q$, where $p$ is a prime number.

\makeatletter
\renewcommand{\@biblabel}[1]{[#1]\hfill}
\makeatother

\end{document}